\documentclass[conference]{IEEEtran}

\usepackage{amssymb}
\usepackage{amsmath}

\newcommand{\real}{\mbox{$\mathbb{R}$}}

\newcommand{\U}{\mbox{$\mathcal{U}$}}
\newcommand{\A}{\mbox{$\mathcal{A}$}}

\newcommand{\reach}{\mbox{$\mathcal{R}$}}

\newcommand{\trdeg}{\mbox{\rm trdeg}}
\newcommand{\Ker}{\mbox{\rm Ker}}

\newcommand{\Na}{\mbox{$\mathcal{N}$}}

\newcommand{\sgn}{\mathrm{sgn}}

\newtheorem{theorem}{Theorem}[section]
\newtheorem{assumption}[theorem]{Assumption}

\newtheorem{corollary}[theorem]{Corollary}
\newtheorem{definition}[theorem]{Definition}

\newtheorem{procedure}[theorem]{Procedure}
\newtheorem{proposition}[theorem]{Proposition}

\begin{document}

\title{Local Nash realizations}

\author{\authorblockN{Jana N\v{e}mcov\'{a}}
\authorblockA{Department of Mathematics\\
Institute of Chemical Technology\\
Technick\'{a} 5, 166 28 Prague 6, Czech republic\\
Email: Jana.Nemcova@vscht.cz}
\and
\authorblockN{Mih\'{a}ly Petreczky}
\authorblockA{Department of Computer Science and Automatic Control \\
Ecole des Mines de Douai\\
941, rue Charles Bourseul, 59508 Douai, France\\
Email: mihaly.petreczky@mines-douai.fr}
}

\maketitle

\begin{abstract}
In this paper we investigate realization theory of a class of non-linear systems,
called Nash systems. Nash systems are non-linear systems whose vector fields and
readout maps are analytic semi-algebraic functions.  In this paper we will present
a characterization of minimality in terms of observability and reachability and
show that minimal Nash systems are isomorphic. The results are local in nature, i.e.
they hold only for small time intervals.  The hope is that the presented results
can be extended to hold globally.
\end{abstract}

\IEEEpeerreviewmaketitle

%%%%%%%%%%%%%%%%%%%%%%%%%%%%%%%%%%%%%%%%%%%%%%%%%%%%%%%%%%%%%%%%%%%%%%%%%%%%%%%%
%%%%%%%%%%%%%%%%%%%%%%%%%%%%%%%%%%%%%%%%%%%%%%%%%%%%%%%%%%%%%%%%%%%%%%%%%%%%%%%%
\section{Introduction}

This paper deals with so-called Nash systems, i.e. non-linear systems the right-hand sides of which are
defined by Nash functions. By a Nash function one refers to a semi-algebraic analytic function. Therefore, the class of Nash systems is an extension of the class of polynomial systems and it is a subclass of analytic nonlinear systems. It is an interesting class of systems to study because of their wide use in e.g. systems biology to model metabolic, signaling, and genetic networks. Moreover, this class also allows a constructive description by means of finitely many polynomial equalities and inequalities which leads to the possibility to derive computational methods for control and analysis of these systems.
\par
An example of framework in systems biology which relies on Nash systems is well-known Biochemical Systems Theory, see \cite{savageau:1969a, savageau:1969b,savageau:1976,voit:2000}. In this framework all processes in metabolic and gene-regulatory networks are modeled by products of power-law functions, i.e. sums of products $x_1^{q_1}x_2^{q_2} \cdots x_n^{q_n}$ of state variables $x_1, \dots , x_n$ taken to rational powers ($q_i$, $i=1,\dots,n$ are rational numbers). Such functions are a special case of Nash functions. Note that the values of rational exponents (kinetic orders) in power-law systems can be related to parameters of different rate laws such as for example Michaelis-Menten kinetics, see \cite{voit:2000}. Another example of framework in systems biology which relies on Nash systems is the tendency modeling framework, see \cite{visser:etal:2000}. It extends the power-law framework by combining mass-action and power-law kinetics into tendency kinetics.
\par
In this paper we deal with the properties of Nash systems which are relevant for modeling of biological data. Namely, we deal with observability, reachability and minimality of Nash systems (realizations) which represent data described by a response map. Our approach is based on realization theory for Nash systems, see \cite{cdc,nemcova:petreczky:schuppen:2013}, which is a continuation of the approach followed by \cite{Jak:RealNonlin} for nonlinear systems, \cite{Son:Resp,bartosiewicz:1988} for polynomial systems and the one in \cite{wang:sontag:1992a,bartosiewicz:1987b,ja:2008b,ja:2008c} for rational systems. It is also closely related to the more recent work \cite{bart:time_scale,bartosiewicz:kotta:pawluszewicz:wyrwas:2011}.
This paper is a longer version with all technical details of the paper submitted to the 18th International Conference on Methods and Models in Automation and Robotics in Poland.
\par
We introduce the notion of local Nash realization of a shifted response map
$p$. Let $u$ be an input signal and $p$ a response map.
Informally, a Nash system is a local Nash realization of a response map $p$ shifted by $u$,
if the response of the system to an input $v$
equals the value of $p$ at $uv$, provided that $v$ is defined on a small enough
time interval. Here $uv$ denotes the concatenation of inputs. The map $p$ shifted by $u$ is denoted by $p_u$.
In other words, the values of $p_u$ equal the outputs of
the Nash system, at least on a small enough time interval.
We show the following:
\\
1) There exists a local Nash realization of $p_u$ if and only if
   the transcendence basis of the observation algebra generated by $p$ is finite.
   This condition is analogous to the finite Hankel-rank condition for
   linear, bilinear and analytic systems \cite{isi:nonlin}.
   Moreover, if the transcendence degree of the observation algebra is $n$,
   then $p_u$ has a minimal local Nash realization of dimension $n$.
   Furthermore, the interval on which $u$ is defined can be chosen to be
   arbitrarily small.
\\
2) If $\Sigma$ is a local Nash realization of $p_u$, then
   $\Sigma$ can be transformed, by following the steps of constructive procedures, to a semi-algebraically reachable and
   observable local Nash realization of $p_{uw}$, for some input $w$.
   The interval on which $w$ is defined can be chosen to be arbitrarily small.
\\
3) A local Nash realization of $p_u$ is minimal among all
   local Nash realizations of $p_v$ (with $v$ ranging through all inputs),
   if and only if it is semi-algebraically reachable and semi-algebraically
   observable.
\\
4) If $\Sigma_1$,$\Sigma_2$ are two local Nash realizations of $p_u$, then
   we can restrict them to open subsets of their respective state-spaces such that
   the resulting systems will be isomorphic and they will be local
   Nash realizations of $p_{uw}$, for some input $w$. Moreover, the interval on which $w$
   is defined can be chosen arbitrarily small.

We expect that the obtained results will be useful for deriving analogous global results. For example,
we hope to be able to prove that semi-algebraic reachability and observability are not only necessary, see
Theorem \ref{thm:5.15}, but also sufficient conditions for minimality of Nash realizations.
Thus, the results are expected to be useful in system identification, model reduction, filtering and control design of Nash systems.

The paper is structured as follows.
In Section \ref{sec:reach} and Section \ref{sec:obs} we present the reachability and
observability reduction procedures for local Nash realizations.
In Section \ref{sec:min}, we prove the characterization of minimality and the
existence conditions discussed above.
The basic notions are introduced in Section \ref{sec:notation} and Section \ref{sec:ext}.

%%%%%%%%%%%%%%%%%%%%%%%%%%%%%%%%%%%%%%%%%%%%%%%%%%%%%%%%%%%%%%%%%%%%%%%%%%%%%%%%
%%%%%%%%%%%%%%%%%%%%%%%%%%%%%%%%%%%%%%%%%%%%%%%%%%%%%%%%%%%%%%%%%%%%%%%%%%%%%%%%
\section{Notation and terminology} \label{sec:notation}

Basic notation and terminology of commutative algebra and real algebraic
geometry used in this paper is adopted from \cite{zariski:samuel:1958,kunz:1985,bochnak:coste:roy:1998}.
The notions of Nash systems and their properties are borrowed from \cite{nemcova:petreczky:schuppen:2013}.
Note that we do not state all results in their full generality as they can be stated over arbitrary
real closed field instead of the field $\real$ of real numbers.

If $A$ is an integral domain over $\real$ then the transcendence degree $\trdeg~A$
of $A$ over $\real$ is defined as the transcendence degree over $\real$ of the field $F$
of fractions of $A$ and it equals the greatest number of
algebraically independent elements of $F$ over $\real$.
\par
A subset $X \subseteq \mathbb{R}^{n}$ is called
\emph{semi-algebraic} if it is a set of points of $\real^n$ which
satisfy finitely many polynomial equalities and inequalities, or
if it is a finite union of such sets.
The {\em dimension} of a semi-algebraic set $X$ is given as the
maximal length of chains of prime ideals of the ring of
polynomial functions on $X$.
\par
Let $X_1 \subseteq \real^n$ and $X_2 \subseteq \real^m$ be semi-algebraic
sets. A map $f:X_1 \rightarrow X_2$ is a {\em semi-algebraic map} if
its graph is a semi-algebraic set in $\real^{n+m}$.

\begin{definition}
A {\em Nash function} on a open semi-algebraic set $X \subseteq
\real^n$ is an analytic and semi-algebraic function from $X$ to
$\real$. We denote the ring of Nash functions on $X$ by $\Na(X)$.
\par
We say that $f:X \rightarrow \real^k$ is a Nash
function if for some $f_1,\ldots,f_k \in \Na(X)$ it holds that
$\forall x \in X: f(x)=(f_1(x),\ldots,f_k(x))$.
\end{definition}

Let $\U_{pc}$ denote the set of piecewise-constant inputs $u = (\alpha_1,t_1) \cdots (\alpha_k,t_k)$
with the input values $\alpha_1, \ldots ,\alpha_k \in U\subseteq
\real^m$ and switching times $t_1, \ldots , t_k \in [0,\infty)$. The input $u$
is defined for $t \in [0, T_u := \sum_{i=1}^k t_k ]$. We denote by $e$ the empty input, i.e. $T_e=0$.

A subset $\widetilde{\U_{pc}} \subseteq \U_{pc}$ is called an {\em admissible set of inputs},
if the following holds:
\begin{itemize}
 \item[(i)] $\forall u \in \widetilde{\U_{pc}} ~\forall t\in [0,T_u] : u_{[0,t]} \in \widetilde{\U_{pc}}$,
 \item[(ii)] $\forall u \in \widetilde{\U_{pc}} ~\forall \alpha \in U ~\exists t > 0 : (u)(\alpha,t) \in \widetilde{\U_{pc}}$,
 \item[(iii)] $\forall u = (\alpha_1,t_1) \cdots (\alpha_k,t_k) \in \widetilde{\U_{pc}}  ~\exists \delta >0 ~\forall \overline{t_i} \in [0,t_i + \delta ], i=1, \dots,k : \overline{u} = (\alpha_1,\overline{t_1}) \cdots (\alpha_k,\overline{t_k}) \in \widetilde{\U_{pc}}.$
\end{itemize}

In this paper we restrict our attention only to Nash systems with the
state-spaces defined as open connected semi-algebraic sets instead of Nash submanifolds,
see \cite{nemcova:petreczky:schuppen:2013}.

\begin{definition} \label{nash_system}
A {\em Nash system} $\Sigma$ with an input-space $U \subseteq
\real^m$ and an output-space $\real^r$ is a quadruple
$(X,f,h,x_0)$ where
\begin{itemize}
 \item[(i)] the state-space $X$ is an open connected semi-algebraic subset of $\real^n$,
 \item[(ii)] the dynamics of the system is given by $\dot{x}(t) = f(x(t),u(t))$ for an input $u \in \U_{pc}$,  where $f:X \times U \rightarrow \real^n$ is such that for every input value $\alpha \in U$ the function $f_{\alpha}:X \ni x \mapsto f(x,\alpha) \in \real^n$ is a Nash function,
 \item[(iii)] the output of the system is specified by a Nash function $h : X \rightarrow \real^r$,
 \item[(iv)] $x_0 = x(0) \in X$ is the initial state of $\Sigma$.
\end{itemize}
\end{definition}

Let $\widetilde{\U_{pc}}(\Sigma)$ denote the largest set of all inputs $u \in \U_{pc}$ such that there exists
a unique trajectory $x_{\Sigma}: [ 0,T_u ] \rightarrow X$ of $\Sigma$ corresponding to $u$ and such that
$\widetilde{\U_{pc}}(\Sigma)$ is an admissible set of inputs.
Notice that  $x_{\Sigma}$ is a continuous piecewise-differentiable function such that
$x_{\Sigma}(0;x_0,u) = x_0$ and $\frac{d}{dt} x_{\Sigma}(t;x_0,u) = f(x_{\Sigma}(t;x_0,u),u(t))$ for
$t\in (\sum_{j=0}^i t_j, \sum_{j=0}^{i+1}t_j)$, $i=0,\ldots , k-1$, $t_0 =0$.

Consider a map $p: \widetilde{\U_{pc}} \rightarrow \real$, where $\widetilde{\U_{pc}}$ is an admissible set
of inputs. For any $\alpha_1, \dots ,\alpha_k \in U$, denote by $\mathbb{T}_{\alpha_1,\dots, \alpha_k}$ the set
of all $k$-tuples $(t_1,\ldots,t_k) \in [0,+\infty)^{k}$ such
that $(\alpha_1,t_1)(\alpha_2,t_2)\cdots
(\alpha_k,t_k) \in \widetilde{\U_{pc}}$.
We say that $p$ is a {\em response map} if for all $\alpha_1,\ldots,\alpha_k \in U$, $k > 0$, the map
$p_{\alpha_1,\dots,\alpha_k}: \mathbb{T}_{\alpha_1,\dots,\alpha_k}
\rightarrow \real : (t_1, \dots ,t_k) \mapsto p((\alpha_1,t_1) \cdots (\alpha_k,t_k))$
is analytic. We use the notation $p \in \A(\widetilde{\U_{pc}} \rightarrow
\real)$. A map $p: \widetilde{\U_{pc}} \rightarrow \real^r$ is called a {\em response map} if its components $p_i$, $i=1,\dots,r$
belong to $\A(\widetilde{\U_{pc}} \rightarrow \real)$. Define $A_{obs}(p)$ as the smallest subalgebra of $\A(\widetilde{\U_{pc}} \rightarrow \real)$
which contains $p_1, \dots,p_r$ and which is closed under the derivative operator
$D_{\alpha}$ on $\A(\widetilde{\U_{pc}} \rightarrow \real)$ defined as
$D_{\alpha}f(u) = \frac{d}{dt}f(u(\alpha,t))|_{t=0+}$.

We say that a response map $p: \widetilde{\U_{pc}} \rightarrow
\real^{r}$ is \emph{locally realized} by a Nash system $\Sigma=(X,f,h,x_0)$ if
$p(u)=h(x_{\Sigma}(T_u;x_0,u))$ for all
$u \in \widetilde{\U_{pc}} \cap  \widetilde{\U_{pc}}(\Sigma)$.

\begin{proposition} \label{rem:restrict}
Let $p: \widetilde{\U_{pc}} \rightarrow \real^{r}$ be a response map and let $S \subseteq \widetilde{\U_{pc}}$ be an
admissible set of inputs.
If $p|_{S}=0$, then $p=0$.
Furthermore, $\trdeg~A_{obs}(p) = \trdeg~A_{obs}(p|_{S})$.
\end{proposition}
\begin{proof}
Consider the map $\Psi: (\A(\widetilde{\U_{pc}} \rightarrow \real))^r \rightarrow (\A(S \rightarrow \real))^r$ defined as
$\Psi(p) = p|_{S}$. It is an algebraic isomorphism. Let us sketch the proof of injectivity of $\Psi$.
Assume $p|_{S} = 0$. That is $p(u) = 0$ for all $u \in S$. In particular, from the definition of an admissible set of inputs,
for any integer $k > 0$, $\alpha_1,\ldots,\alpha_k \in U$, there exists $\bar{t}_i > 0$,
$i=1,\ldots,k$ such that for all $t_i \in [0,\bar{t}_i]$, $i=1,\ldots,k$,
the input $v=(\alpha_1,t_1)\cdots (\alpha_k,t_k)$ satisfies
$v \in S$ and $p(v)=0$. From the analyticity of $p$ it follows that $p=0$.
\par
Then, $\Psi(A_{obs}(p)) = A_{obs}(p|_{S})$ which implies $\trdeg~A_{obs}(p) = \trdeg~A_{obs}(p|_{S})$.
\end{proof}

Consider a response map $p: \widetilde{\U_{pc}} \rightarrow \real^{r}$ and $u \in \widetilde{\U_{pc}}$. Denote by $p_u$ the map
$p_u:\widetilde{\U_{pc}}^{u} \rightarrow \real^r$, where
$\widetilde{\U_{pc}}^{u}=\{v \in \widetilde{\U_{pc}} \mid uv \in \widetilde{\U_{pc}}\}$, such that
$\forall v \in \widetilde{\U_{pc}}^u\ \ :  p_u(v) = p(uv)$.
It is easy to see that $p_u$ is again a response map.

\begin{definition} \label{def:obs}
We say a Nash system $\Sigma = (X,f,h,x_0)$ is {\em
semi-algebraically observable} if $\trdeg~A_{obs}(\Sigma)=\trdeg~\Na(X)$.
Observation algebra $A_{obs}(\Sigma)$ is defined as the smallest subalgebra of $\Na(X)$
which contains all components of $h$ and which is closed under taking Lie derivatives
with respect to vector fields $f_{\alpha}$, $\alpha \in U$.
\end{definition}

\begin{definition} \label{def:algreachable}
We say a Nash system $\Sigma = (X,f,h,x_0)$ is
\emph{semi-algebraically reachable} if
\[\forall g \in \Na(X) : \left( g = 0 ~\mbox{on}~ \reach(x_0) \Rightarrow g = 0 \right), \]
where $\reach(x_0)$ denotes the set of states of $\Sigma$
reachable from $x_0$ by the inputs of $\widetilde{\U_{pc}}(\Sigma)$, i.e.
$\reach(x_0) = \{x_{\Sigma}(T_u;x_0,u) \mid u \in \widetilde{\U_{pc}}(\Sigma)\}$.
\end{definition}

This definition differs from the one introduced in \cite{nemcova:petreczky:schuppen:2013}. There
we consider the inputs of $\widetilde{\U_{pc}} \subseteq \widetilde{\U_{pc}}(\Sigma)$ instead
of $\widetilde{\U_{pc}}(\Sigma)$ to define reachable states. However, Proposition \ref{prop2} stated
below yields that Definition \ref{def:algreachable} and the corresponding definition in \cite{nemcova:petreczky:schuppen:2013} are the same.

Let $\Sigma$ be a local Nash realization of a response map
$p:\widetilde{\U_{pc}} \rightarrow \mathbb{R}^r$. Define
the \emph{state-to-output map} $\tau_{\Sigma,p}^{*}:\Na(X) \rightarrow \A(\widetilde{\U_{pc}} \cap \widetilde{\U_{pc}}(\Sigma) \rightarrow \real)$
as follows: for any $g \in \Na(X)$, and any
$u \in \widetilde{\U_{pc}} \cap \widetilde{\U_{pc}}(\Sigma)$,
$\tau_{\Sigma,p}^{*}(g)=g(x_{\Sigma}(T_u;x_0,u))$.
The map $\tau_{\Sigma,p}^{*}$ is closely related to semi-algebraic reachability and
observability of $\Sigma$.

\begin{assumption}
In the sequel, we will assume that the set $U$ of input values is finite.
\end{assumption}

%%%%%%%%%%%%%%%%%%%%%%%%%%%%%%%%%%%%%%%%%%%%%%%%%%%%%%%%%%%%%%%%%%%%%%%%%%%%%%%%
%%%%%%%%%%%%%%%%%%%%%%%%%%%%%%%%%%%%%%%%%%%%%%%%%%%%%%%%%%%%%%%%%%%%%%%%%%%%%%%%
\section{Extension to negative times} \label{sec:ext}

In this section we define the extension of
response maps to negative switching times. Notice that if a response map $p$ has a
realization by a Nash system $\Sigma$, then one can extend $p$ to act on piecewise-constant
inputs where the duration of inputs is allowed to be negative. More precisely,
let $\Sigma=(X,f,h,x_0)$ be a Nash system.
For any $\alpha \in U$,
$t \in \mathbb{R}$, denote by $f_{\alpha}^t(x)$ the flow at time $t$ of the vector field
$f_{\alpha}$ from the state $x$.
If $z=f^{t}_{\alpha}(x)$, then $f^{-t}_{\alpha}(z)$ is well defined.

Let $\U_{pc}^{-}$ denote the set of all sequences of the form
$u=(\alpha_1,t_1)\cdots (\alpha_k,t_k)$, where $\alpha_1,\ldots,\alpha_k \in U$ and
$t_1,\ldots,t_k \in \mathbb{R}$. We call such sequences \emph{generalized inputs}.
If $u, v \in \U_{pc}^{-}$ then $uv$ denotes concatenation of $u$ and $v$.
By $\stackrel{\leftarrow}{u}$ we denote the sequence
$(\alpha_k,-t_k)\cdots (\alpha_1,-t_1)$. Notice that $u \stackrel{\leftarrow}{u} = \stackrel{\leftarrow}{u} u = e \in \U_{pc}^{-}$.

A set $\overline{\U_{pc}} \subseteq \U^{-}_{pc}$ is said to be an {\em admissible set of generalized inputs}, if the following holds:
\begin{itemize}
 \item[(i)] $\forall u = (\alpha_1,t_1)\cdots (\alpha_k,t_k) \in \overline{\U_{pc}} ~\forall i\in \{1,\dots,k\} \, \forall t\in I(0,t_i): (\alpha_1,t_1)\cdots (\alpha_{i-1},t_{i-1})(\alpha_i,t) \in \overline{\U_{pc}}$,
 \item[(ii)] $\forall u \in \overline{\U_{pc}} \,\forall \alpha \in U \,\exists \delta > 0 \,\forall t \in (-\delta,\delta): (u)(\alpha,t) \in \overline{\U_{pc}}$,
 \item[(iii)] $\forall u = (\alpha_1,t_1) \cdots (\alpha_k,t_k) \in \overline{\U_{pc}}  ~\exists \delta >0 ~\forall \overline{t_i} \in I(0,(\sgn \,t_i)(|t_i| + \delta)), i=1, \dots,k : \overline{u} = (\alpha_1,\overline{t_1}) \cdots (\alpha_k,\overline{t_k}) \in \overline{\U_{pc}},$
\end{itemize}
where $I(0,a)$ denotes either the interval $[0,a]$ for $a\geq 0$ or the interval $[a,0]$ for $a< 0$.

By $\overline{\U_{pc}}(\Sigma)$ we denote the largest admissible set of inputs of $\U_{pc}^-$ such that $\forall u =(\alpha_1,t_1)\cdots (\alpha_k,t_k) \in \U_{pc}^{-} : x_{\Sigma}(T_u;x_0,u) := f_{\alpha_k}^{t_k} \circ \cdots \circ f_{\alpha_1}^{t_1}(x_0)$ is well defined.

For every $\alpha_1,\ldots,\alpha_k \in U$, denote by
$\overline{\mathbb{T}}_{\alpha_1,\ldots,\alpha_k}$ the set of all $(t_1,\ldots,t_k) \in \mathbb{R}^k$,
such that $(\alpha_1,t_1)\cdots (\alpha_k,t_k) \in \overline{\U_{pc}}$.
We say that a map $\bar{p}:\overline{\U_{pc}} \rightarrow \real$ is a \emph{generalized response map} if
\begin{itemize}
\item[(i)]
 $\bar{p}_{\alpha_1,\ldots,\alpha_k}: \overline{\mathbb{T}}_{\alpha_1,\ldots,\alpha_k} \rightarrow \real : (t_1,\ldots,t_k) \mapsto \bar{p}((\alpha_1,t_1)\cdots (\alpha_k,t_k))$ is analytic,
\item[(ii)]
  for all $u,v$ and for any $\alpha \in U$,
  $\bar{p}(u(\alpha,0)v)=\bar{p}(uv)$ and
  $\bar{p}(u(\alpha,t_1)(\alpha,t_2)v)=\bar{p}(u(\alpha,t_1+t_2)v)$,
  provided $u(\alpha,0)v,u(\alpha,t_1)(\alpha,t_2)v,u(\alpha,t_1+t_2)v,uv \in \overline{\U_{pc}}$.
\end{itemize}
We say that a map $\bar{p}:\overline{\U_{pc}} \rightarrow \real^r$ is a generalized response map,
if all its components $\bar{p}_1,\ldots,\bar{p}_r$ are generalized response maps according to the definition above.

Denote by $\A(\overline{\U_{pc}} \rightarrow \real)$ the set of all generalized response maps on $\overline{\U_{pc}}$.
It is easy to see that $\A(\overline{\U_{pc}} \rightarrow \real)$ is an algebra
with respect to point-wise addition and multiplication. Similarly to \cite{ja:2008b}
it can be shown that $\A(\overline{\U_{pc}} \rightarrow \real)$ is an integral domain.

\begin{definition} \label{def:negext}
Let $\widetilde{\U_{pc}}$ be an admissible set of inputs and
let $p:\widetilde{\U_{pc}} \rightarrow \real^r$ be a response map.
We say that $p$ has an \emph{extension to negative times}, if there exists an
admissible set of generalized inputs $\overline{\U_{pc}}$ and a generalized response map $\bar{p}:\overline{\U_{pc}} \rightarrow \real^r$ such that
\\
1) $\widetilde{\U_{pc}} \subseteq \overline{\U_{pc}}$, and 2) the restriction of $\bar{p}$ to $\widetilde{\U_{pc}}$ equals $p$.
\end{definition}

\begin{assumption}\label{assum0}
In the sequel, we will assume that the response map $p$ has an extension to negative times.
\end{assumption}
Assumption \ref{assum0} is not restrictive: if $p$ has a local realization by a  Nash system
$\Sigma$, then set $\overline{\U_{pc}} := \overline{\U_{pc}}(\Sigma)$ and define
$\bar{p}(u)=h(x_{\Sigma}(T_u;x_0,u))$ for all $u \in \overline{\U_{pc}}$.
It is then clear that $\bar{p}|_{\widetilde{\U_{pc}}}$ satisfies Definition \ref{def:negext}.

Assume now that $p:\widetilde{\U_{pc}} \rightarrow \real^r$ has an extension $\bar{p}$ to negative times.
Let $\overline{\U_{pc}}$ be the domain of the map
$\bar{p}$.
Define the map $\Psi:\A(\overline{\U_{pc}} \rightarrow \real) \rightarrow \A(\widetilde{\U_{pc}} \rightarrow \real)$ such that $\Psi(g)$ is the restriction of $g$ to the set
$\widetilde{\U_{pc}}$. It is clear that $\Psi$ is an injective algebra morphism.
Define the derivative operator $D_{\alpha}$ on $\A(\overline{\U_{pc}} \rightarrow \real)$
as $D_{\alpha} f(u)=\frac{d}{dt} f(u(\alpha,t))|_{t=0}$.
Define $A_{obs}(\bar{p})$ as the smallest subalgebra of
$\A(\overline{\U_{pc}} \rightarrow \real)$ which contains $\bar{p}_1,\ldots,\bar{p}_d$
and which is closed under the operator $D_{\alpha}$.
It then follows that $\Psi(A_{obs}(\bar{p}))=A_{obs}(p)$ and
$\Psi:A_{obs}(\bar{p}) \rightarrow A_{obs}(p)$ is an isomorphism.
In particular, $\trdeg ~A_{obs}(p)=\trdeg ~A_{obs}(\bar{p})$.

For any $u \in \overline{\U_{pc}}$, consider the map
$\Phi_u:\A(\overline{\U_{pc}} \rightarrow \real) \rightarrow \A(\overline{\U_{pc}}^{u} \rightarrow \real)$ defined as
$\Phi_u(f)(v)=f(uv)$, where $\overline{\U_{pc}}^{u} = \{v \in \U_{pc}^{-} \mid uv \in \overline{\U_{pc}}\}$.
It is clear that $\Phi_u$ is an algebraic map.
Moreover, $D_{\alpha} \circ \Phi_u=\Phi_u \circ D_{\alpha}$ and thus
$\Psi \circ \Phi_u \circ \Psi^{-1} (A_{obs}(p))=A_{obs}(p_u)$.
It can be shown that $\Phi_u$ is injective, see \cite{ArxivePaper}.
From this it follows that $\Psi \circ \Phi_u \circ \Psi^{-1}:A_{obs}(p) \rightarrow A_{obs}(p_u)$ is an
algebra isomorphism.
Indeed, if $\Phi_u(f)=0$, then $f(u\stackrel{\leftarrow}{u}v)=f(v)=0$ for all
$v \in \overline{\U_{pc}}^{u}$. Notice that for any $\alpha_1,\ldots,\alpha_k \in U$,
there exists an open subset $V \subseteq \real^k$ such that
$(\alpha_1,\tau_1)\cdots (\alpha_k,\tau_k) \in \overline{\U_{pc}}$ for all
$(\tau_1,\ldots,\tau_k) \in V$. Since the map
$f_{\alpha_1,\ldots,\alpha_k}(t_1,\ldots,t_k)=f((\alpha_1,t_1)\cdots (\alpha_k,t_k))$ is
analytic in $t_1,\ldots,t_k$ and $f_{\alpha_1,\ldots,\alpha_k}$ is zero on an open
subset of $V$, and the domain of $f_{\alpha_1,\ldots,\alpha_k}$ is connected, it
follows that $f_{\alpha_1,\ldots,\alpha_k}=0$. Since $\alpha_1,\ldots,\alpha_k$ were
arbitrary, it then implies that $f$ is zero, i.e. $\Ker ~\Phi_u=\{0\}$.
The discussion above is summarized in the following corollary.

\begin{corollary} \label{col1}
Let $p:\widetilde{\U_{pc}} \rightarrow \real^r$ be a response map. For any $u \in \widetilde{\U_{pc}}$, $A_{obs}(p)$ and $A_{obs}(p_u)$ are isomorphic.
In particular, $\trdeg~ A_{obs}(p)=\trdeg~ A_{obs}(p_u)$.
\end{corollary}

%%%%%%%%%%%%%%%%%%%%%%%%%%%%%%%%%%%%%%%%%%%%%%%%%%%%%%%%%%%%%%%%%%%%%%%%%%%%%%%%
%%%%%%%%%%%%%%%%%%%%%%%%%%%%%%%%%%%%%%%%%%%%%%%%%%%%%%%%%%%%%%%%%%%%%%%%%%%%%%%%
\section{Reachability reduction} \label{sec:reach}

\begin{procedure} \label{proc:reach}
\emph{
Let $\Sigma = (X \subseteq \real^n,f,h,x_0)$ be a local Nash realization
of $p:\widetilde{\U_{pc}} \rightarrow \real^r$ and let $\epsilon >0$ be arbitrary.
\begin{enumerate}
\item Consider the ideal $I$ of functions of $\Na(X)$ which vanish on $\reach(x_0)$.
      Choose a transcendence basis $\varphi_1, \ldots , \varphi_d \in \Na(X)$ of $\Na(X)/I$. Denote $\Phi = (\varphi_1, \ldots , \varphi_d)$.
\item Let $\pi_i$, $i=1,\ldots, n$ be the coordinate functions on $X$. Find non-zero polynomials $Q_i \in \real[T,T_1,\ldots, T_d]$, $i=1,\ldots,n$ with the smallest possible degree and an input $u \in \widetilde{\U_{pc}}(\Sigma) \cap \widetilde{\U_{pc}}$ such that $T_u < \epsilon$, $Q_i(\pi_i, \Phi) = 0$ on $\reach(x_0)$ and $\frac{d}{dt} Q_i(\pi_i(x_{\Sigma}(T_u;x_0,u)), \Phi(x_{\Sigma}(T_u;x_0,u) ) )  \neq 0$.
\item By applying implicit function theorem, find open semi-algebraic subsets $V \subseteq \real^d$, and
  $W \subseteq \real^n$ and a Nash map $G:V \rightarrow W$ such that
\\(1) $\Phi(x_{\Sigma}(T_u;x_0,u)) \in V$ and $x_{\Sigma}(T_u;x_0,u) \in W$,
\\(2) $G(\Phi(x_{\Sigma}(T_{uv},x_0,uv))) = x_{\Sigma}(T_{uv},x_0,uv)$
  for every $v$ such that for all $t \in [0,T_v]$,
  $x_{\Sigma}(T_{u}+t;x_0,uv) \in W$.
% $T_v$ is small enough and $uv \in \widetilde{\U_{pc}}(\Sigma)$.
\item Then the system $\Sigma_V = (V,f^V,h^V,x_0^V)$, where $f^V_{\alpha}(z)=D\Phi(G(z))f_{\alpha}(G(z))$, $h^V(z)=h(G(z))$ for $z \in V$ and $x^{V}_0=\Phi(x_{\Sigma}(T_u;x_0,u))$, is a semi-algebraically reachable local Nash realization of $p_u$.
\end{enumerate}
}
\end{procedure}
To prove the correctness of the procedure, let us first present the following
alternative characterization of semi-algebraic reachability.

\begin{proposition} \label{reach:prop1}
A Nash system $\Sigma = (X,f,h,x_0)$ is semi-algebraically reachable if and only if
$\dim X = \trdeg ~\Na(X)/I$, where $I \subseteq \Na(X)$ is the ideal of functions which vanish on $\reach(x_0)$.
The ideal $I$ is prime.
\end{proposition}
\begin{proof}
We first show that $I$ is a prime ideal. To this end, recall from
Section \ref{sec:notation} the definition of the map
$\tau_{\Sigma}^{*}$ which maps every $g \in \Na(X)$ to the
response map $\tau_{\Sigma}^{*}(g):\widetilde{\U_{pc}}(\Sigma) \rightarrow \real$
defined as $\tau_{\Sigma}^{*}(g)(u)=g(x_{\Sigma}(T_u;x_0,u))$.
Recall that $\widetilde{\U_{pc}}(\Sigma)$ is an admissible set of inputs and hence
by \cite{ja:2008b}, the set of all response maps
$\A(\widetilde{\U_{pc}}(\Sigma) \rightarrow  \real)$ is an integral domain.
Notice that $\tau_{\Sigma}^{*}:\Na(X) \rightarrow \A(\widetilde{\U_{pc}}(\Sigma) \rightarrow \real)$ is an algebra morphism.
It is easy to see that $\Ker ~\tau_{\Sigma}^{*}=I$. Since $\tau_{\Sigma}^{*}$ is an algebraic morphism
between integral domains, its kernel must be a prime ideal: indeed, if
$g_1g_1 \in \Ker ~\tau_{\Sigma}^{*}=I$, then
$\tau_{\Sigma}^{*}(g_1g_2)=\tau_{\Sigma}^{*}(g_1)\tau_{\Sigma}^{*}(g_2)=0$, from
which it follows that either $\tau_{\Sigma}^{*}(g_1)=0$ or
$\tau_{\Sigma}^{*}(g_2)=0$, i.e. $g_1 \in I$ or $g_2 \in I$.

If $\Sigma$ is semi-algebraically reachable then clearly $I=\{0\}$ and thus $\dim X=\trdeg ~\Na(X)$.

Assume that $\dim X = \trdeg ~\Na(X)/I$. Notice that
$\Na(X)/I$ is an algebra which is isomorphic to
$\tau_{\Sigma}^{*}(\Na(X))$.
Hence, $\dim X=\trdeg~ \Na(X)=\trdeg~ \tau_{\Sigma}^{*}(\Na(X))$.
Then from \cite[Chapter II, Theorem 28,29]{zariski:samuel:1958} it follows that $\tau_{\Sigma}^{*}$ is injective, i.e. $I=\Ker ~\tau_{\Sigma}^{*}=\{0\}$.
\end{proof}

From the analyticity of Nash functions one derives the following:
\begin{proposition} \label{prop2}
Let $\Sigma = (X,f,h,x_0)$ be a Nash system. Let $\widetilde{\U_{pc}}$ be an admissible set of inputs
such that $\widetilde{\U_{pc}} \subseteq \overline{\U_{pc}}(\Sigma)$. Then for any $g \in \Na(X)$
such that $\forall u \in \widetilde{\U_{pc}}: g(x_{\Sigma}(T_u;x_0,u))=0$ it holds that
$\forall u \in \overline{\U_{pc}}(\Sigma): g(x_{\Sigma}(T_u;x_0,u))=0$.
\end{proposition}
Proposition \ref{prop2} and Proposition \ref{prop:7.1} yield the following
corollary.

\begin{corollary}
\label{prop2:col}
Let $\Sigma$ be a local realization of a response map $p$.
Then $\Sigma$ is semi-algebraically reachable if and only if the state-to-output map
$\tau_{\Sigma,p}^{*}$ is injective.
\end{corollary}
\begin{theorem} \label{theorem:reach}
Procedure \ref{proc:reach} is correct. Namely, by following the steps of Procedure \ref{proc:reach} one transforms any local realization $\Sigma=(X,f,h,x_0)$ of a response map $p$
to a semi-algebraically reachable local realization of $p_u$, where $u \in \widetilde{\U_{pc}} \cap \widetilde{\U_{pc}}(\Sigma)$ can be chosen such that $T_u < \epsilon$ for any $\epsilon > 0$.
\end{theorem}
\begin{proof}
Let $I, \varphi_1,\ldots,\varphi_d$ be as in step 1). First we
prove the existence of polynomials $Q_i$ in step 2).
For any
$g \in \Na(X)$, there exists a non-zero polynomial $Q$ such that
$Q(g,\Phi)$ equals zero on $\reach(x_0)$.
In particular, for the coordinate functions $\pi_1,\ldots,\pi_n$ on $X$
there exist non-zero polynomials $Q_i$, $i=1,\ldots,n$,
such that $Q_i(\pi_i(x_{\Sigma}(T_u;x_0,u)),\Phi(x_{\Sigma}(T_u;x_0,u))) =0$
for all $u \in \widetilde{\U_{pc}}(\Sigma)$.
Choose the polynomials $Q_i(T,T_1,\ldots,T_d)$ in such a way that their degrees are the minimal ones.
Hence, $\frac{dQ_i}{dT}(\pi_i,\Phi)$ is not identically zero on
$\reach(x_0)$.
For any $\epsilon > 0$, we can find $u \in \widetilde{\U_{pc}}$ such
that $T_u < \epsilon$ and
$\frac{dQ_i}{dT}(\pi_i(x_{\Sigma}(T_u;x_0,u)),\Phi(x_{\Sigma}(T_u;x_0,u))) \ne 0$.
Let us fix such $u$ for the given $\epsilon$.

In step 3) we apply the implicit function theorem \cite[Corollary 2.9.8]{bochnak:coste:roy:1998}
to each $Q_i(\pi_i,\Phi)$ on an open neighborhood of $x_{\Sigma}(T_u;x_0,u)$.
We obtain open connected semi-algebraic sets $V_i \subseteq \mathbb{R}^{d}$, $W_i \subseteq \mathbb{R}$ and Nash maps
$G_i: V_i \rightarrow W_i$ by which we define non-empty semi-algebraic sets $V:=\bigcap_{i=1}^n V_i$,
$W:=W_1 \times \cdots \times W_n$ and the Nash map $G=(G_1,\ldots,G_n)$.
Notice that $G$ is the unique Nash map such that for any
$(p,q) \in V \times W$ it holds that $Q_i(q_i,p)=0$, $i=1,\ldots,d$ is equivalent to $q=G(p)$.
In particular, for any $z \in W$ such that $z$ is reachable,
$G(\Phi(z))=z$.

Let us prove that $\Sigma_V$ defined in step 4) is a local realization of $p_u$.
Let $S$ be the largest set of admissible inputs such that
$\forall v \in S \forall t \in [0,T_v] : x_{\Sigma}(T_u+t;x_0,uv) \in W$.
It is easy to see that for any $v \in S$,
$t \mapsto z(t)= \Phi(x_{\Sigma}(T_u+t;x_0,uv))$ satisfies
$z(t) \in V$ and $\dot z(t)=f^{V}_{v(t)}(z(t))$, i.e.
$z(t)=x_{\Sigma_V}(t;x_0^V,v)$.
Hence, $S \subseteq \widetilde{\U_{pc}} \cap \widetilde{\U_{pc}}(\Sigma_V)$.
%Notice that
%$\frac{d}{dt} x_{\Sigma}(T_{u}+t;x_0,uv) = f_{\alpha}(x_{\Sigma}(T_u+t;x_0,uv))$.
%But $z(t)=\Phi(x_{\Sigma}(T_u+t;x_0,uv)) = \Phi(G(z(t))$ and hence
%$\dot z(t) = \frac{d}{dt} \Phi(G(z(t))) = D\Phi(G(z(t))) \frac{d}{dt} G(z(t)) =
%D\Phi(G(z(t))) f_{\alpha}(G(z(t))) = f^{V}_{\alpha}(G(z(t)))$.
Finally, from the definition of $h^V$, it follows that
$h^V(z(t)) = h(G(z(t))) = h(x_{\Sigma}(T_u+t;x_0,uv))$.
Hence, for any $v \in S$,
$p(uv)=h^{V}(x_{\Sigma_V}(T_v;x_0^V,v))$.
Define the map $\hat{p}:\widetilde{\U_{pc}}^{u} \cap \widetilde{\U_{pc}}(\Sigma_V) \rightarrow \real^r$ as $\hat{p}(v)=h^{V}(x_{\Sigma_V}(T_v;x_0^V,v))$.
Then $(\hat{p}-p_u)|_{S}=0$. Since $S$ is an admissible set of inputs,
Proposition \ref{rem:restrict} implies that $\hat{p}=p$, i.e.
$\Sigma_V$ is a local realization of $p_u$.

It is left to show that $\Sigma_V$ is semi-algebraically reachable. Assume now that there exists a non-zero polynomial $Q \in \real[T_1,\ldots,T_d]$ such that
$Q(x_{\Sigma_V}(T_v;x_0^V,v))=0$ for all $v \in \overline{\U_{pc}}(\Sigma_V)$.
Because $x_{\Sigma_V}(T_v;x_0^V,v) = \Phi(x_{\Sigma}(T_{uv};x_0,uv))$ for
all $v \in \overline{\U_{pc}}(\Sigma_V)$, it follows that
$Q(\Phi(x_{\Sigma}(T_{uv};x_0,uv)))=0$ for all $v$ such that
$uv \in \overline{\U_{pc}}(\Sigma)$. Notice that for any
$v \in \overline{\U_{pc}}(\Sigma)$, $u\stackrel{\leftarrow}{u}v \in \overline{\U_{pc}}(\Sigma)$, from which it follows that
$Q(\Phi(x_{\Sigma}(T_{v};x_0,v)) = Q(\varphi_1(x_{\Sigma}(T_{v};x_0,v)), \ldots , \varphi_d(x_{\Sigma}(T_{v};x_0,v))) = 0$ for all $v \in \overline{\U_{pc}}(\Sigma)$.
The latter implies that $Q(\varphi_1, \ldots , \varphi_d)=0$ on $\overline{\reach}(x_0) = \{x_{\Sigma}(T_u;x_0,u)|u\in \overline{\U_{pc}}(\Sigma)\}$ and thus on $\reach(x_0)$.
Therefore, $Q(\varphi_1, \ldots , \varphi_d) \in I$. But this means that $\varphi_1,\ldots,\varphi_d$ are algebraically dependent, which is
a contradiction.
Hence, if $I_V$ is the ideal of functions from $\Na(V)$ which
vanish on  $\reach(x_0^V)$, then $\trdeg ~\Na(V)= \trdeg ~\Na(V)/I_V$. Then, from Proposition \ref{reach:prop1},
$\Sigma_V$ is semi-algebraically reachable.
\end{proof}

%%%%%%%%%%%%%%%%%%%%%%%%%%%%%%%%%%%%%%%%%%%%%%%%%%%%%%%%%%%%%%%%%%%%%%%%%%%%%%%%
%%%%%%%%%%%%%%%%%%%%%%%%%%%%%%%%%%%%%%%%%%%%%%%%%%%%%%%%%%%%%%%%%%%%%%%%%%%%%%%%
\section{Observability reduction} \label{sec:obs}

\begin{procedure} \label{proc:obs}
\emph{
Let $\Sigma=(X,f,h,x_0)$ be a semi-algebraically reachable local realization of a response map $p: \widetilde{\U_{pc}} \rightarrow \real^r$ and let $\epsilon >0$ be arbitrary.
\begin{enumerate}
\item Choose a transcendence basis $\varphi_1,\ldots,\varphi_d$ of $A_{obs}(\Sigma)$. Let us denote $\Phi = (\varphi_1,\ldots,\varphi_d)$.
\item Find non-zero polynomials $Q_{\alpha,i}, P_j \in \real[T,T_1,\ldots,T_d]$ with the smallest possible degree such that $Q_{\alpha,i}(f_{\alpha}\varphi_i(x),\Phi(x))=0$ and $P_j(h_j(x),\Phi(x))=0$ for all $x \in X$, $i=1,\ldots,d$, $\alpha \in U$, $j=1,\dots,r$.
\item Find $u \in \widetilde{\U_{pc}}(\Sigma) \cap \widetilde{\U_{pc}}$ such that $T_u < \epsilon$ and such that $\frac{d}{dT} Q_{\alpha,i}(f_{\alpha}\varphi_i,\Phi)$ and $\frac{d}{dT} P_{j}(h_j,\Phi)$ do not vanish on an open neighborhood of $\overline{x} := x_{\Sigma}(T_u;x_0,u)$.
\item By implicit function theorem, find
\\
(1) open connected semi-algebraic subsets $W_{\alpha,i} \subseteq \mathbb{R}$, $V_{\alpha,i} \subseteq \mathbb{R}^d$ and a Nash map $\psi_{\alpha,i}:V_{\alpha,i} \rightarrow W_{\alpha,i}$ such that $f_{\alpha}\varphi_i(\overline{x}) \in W_{\alpha,i}$, $\Phi(\overline{x}) \in V_{\alpha,i}$ and $\psi_{\alpha,i}(\Phi(\overline{x}))= f_{\alpha}\varphi_i(\overline{x})$, and for all $(f_{\alpha}\varphi_i(x),\Phi(x)) \in W_{\alpha,i} \times V_{\alpha,i}$,
$Q_{\alpha,i}(f_{\alpha}\varphi_i(x),\Phi(x))=0 \iff f_{\alpha}\varphi_i(x) = \psi_{\alpha,i}(\Phi(x))$,
\\
(2) open connected semi-algebraic subsets $W_{j} \subseteq \mathbb{R}$, $V_{j} \subseteq \mathbb{R}^d$ and a Nash map $\psi_{j}:V_{j} \rightarrow W_{j}$ such that $h_j(\overline{x}) \in W_{j}$, $\Phi(\overline{x}) \in V_{j}$ and $\psi_{j}(\Phi(\overline{x}))= h_j(\overline{x})$, and for all $(h_j(x),\Phi(x)) \in W_{j} \times V_{j}$,
$P_{j}(h_j(x),\Phi(x))=0 \iff h_j(x) = \psi_{j}(\Phi(x))$.
\item Define $V:=\bigcap_{\substack{i=1,\ldots, d \\ \alpha \in U \\j=1,\ldots, r}}V_{\alpha,i}\cap V_j$.
\item Then the system $\Sigma_V = (V,f^V,h^V,x_0^V)$, where $x_0^V = \Phi(\overline{x})$, $f^V=\{f^{V}_{\alpha} \mid \alpha \in U\}$ is such that $f^{V}_{\alpha}(x) = (\psi_{\alpha,1}(x) , \ldots ,\psi_{\alpha,d}(x))$ for all $\alpha \in U$, $x\in V$, and $h^{V}(\Phi(x)) = h(x)$, is a local Nash realization of $p_u$ which is both semi-algebraically observable and semi-algebraically reachable.
\end{enumerate}
}
\end{procedure}

\begin{theorem} \label{thm:obs1}
Procedure \ref{proc:obs} is correct. Namely, by following the steps of Procedure \ref{proc:obs} one transforms any semi-algebraically reachable local realization $\Sigma=(X,f,h,x_0)$ of a response map $p$ with finite set of input values to a semi-algebraically reachable and semi-algebraically observable local realization of $p_u$, where $u \in \widetilde{\U_{pc}} \cap \widetilde{\U_{pc}}(\Sigma)$ can be chosen such that $T_u < \epsilon$ for any $\epsilon > 0$.
\end{theorem}
\begin{proof}
Let $\varphi_1,\ldots,\varphi_d$ be a transcendence basis of $A_{obs}(\Sigma)$ constructed in step 1).
Then $f_{\alpha}\varphi_i, h_j$, where $i=1,\dots,d$, $\alpha \in U$, $j=1,\dots,r$, are algebraic over $\{\varphi_1,\ldots,\varphi_d\}$. Therefore, there exist non-zero polynomials $Q_{\alpha,1},\ldots,Q_{\alpha,d}, P_1, \dots ,P_r \in \real[T,T_1,\ldots,T_d]$, to be found in step 2), such that $Q_{\alpha,i}(f_{\alpha}\varphi_i(x),\Phi(x))=0$ and $P_j(h_j(x),\Phi(x))=0$ for all $x \in X$, $i=1,\ldots,d$, $\alpha \in U$, $j=1,\dots,r$. Note that the maps $X \ni x \mapsto Q_{\alpha,i}(f_{\alpha}\varphi_i(x),\Phi(x))$ and $X \ni x \mapsto P_j(h_j(x),\Phi(x))$ belong to $\Na(X)$.
\par
Let us prove the existence of an input $u$ required in step 3).
By choosing $Q_{\alpha,1},\ldots,Q_{\alpha,d}$ and $ P_1, \dots ,P_r$ to be of minimal degree with respect to the indeterminant $T$, we can assume that $\frac{d}{dT} Q_{\alpha,i}(f_{\alpha}\varphi_i,\Phi) \ne 0$, $i=1,\ldots, d$, $\alpha \in U$ and $\frac{d}{dT} P_{j}(h_j,\Phi) \ne 0$, $j=1,\ldots, r$.
Then, there exist $\overline{x} \in X$ and an open neighborhood $O$ of $\overline{x}$
such that $\frac{d}{dT} Q_{\alpha,i}(f_{\alpha}\varphi_i,\Phi)$ and $\frac{d}{dT} P_{j}(h_j,\Phi)$ do not vanish on $O$.
Let us assume for contradiction that such $O$ does not exist. Then $X$ can be expressed as finite union of the zero sets of $\frac{d}{dT} Q_{\alpha,i}(f_{\alpha}\varphi_i,\Phi)$ and $\frac{d}{dT} P_{j}(h_j,\Phi)$ which are nowhere dense sets. Therefore $X$ is finite union of sets of the first category. However, $X$ is a non-empty open semi-algebraic subset of $\real^n$ which means it is a topological manifold and thus a locally compact Hausdorff space. Therefore, by Baire's theorem, see \cite[Theorem 2.2]{rudin:1973}, $X$ is of the second category which means that $X$ cannot be expressed as countable union of sets of the first category.
Let us choose $\overline{x}$ from above as an element of $\reach(x_0)$. Since $O \cap \reach(x_0) \neq \emptyset$, such $\overline{x}$ does exist. Let us assume that $O \cap \reach(x_0) = \emptyset$ for a contradiction.
Then, the Nash function
\[ R:=\prod_{\substack{i=1,\ldots, d \\ \alpha \in U \\j=1,\ldots, r}} \frac{d}{dT} Q_{\alpha,i}(f_{\alpha}\varphi_i,\Phi) \frac{d}{dT} P_{j}(h_j,\Phi)\]
is zero on $\reach(x_0)$. Because $R \in \Na(X)$ and because $\Sigma$ is semi-algebraically reachable, it implies that $R = 0$ on X. Therefore, since $\Na(X)$ is an integral domain, at least one of the functions in the product would have to be zero on $X$. This contradicts the fact that $\frac{d}{dT} Q_{\alpha,i}(f_{\alpha}\varphi_i,\Phi) \ne 0$ and $\frac{d}{dT} P_{j}(h_j,\Phi) \ne 0$.
Then $u$ required in step 3) is the input $u \in \widetilde{\U_{pc}}(\Sigma)$ such that
$\overline{x} = x_{\Sigma}(T_u;x_0,u)$.
\par
In step 4) we apply the implicit function theorem \cite[Corollary 2.9.8]{bochnak:coste:roy:1998} to each of the functions $Q_{\alpha,i}(f_{\alpha}\varphi_i,\Phi)$ and $P_{j}(h_j,\Phi)$ on the open neighborhood $O$ of $\overline{x}$. We obtain the  sets $W_{\alpha,i}, V_{\alpha,i},W_{j}, V_{j}$ and the Nash maps $\psi_{\alpha,i}, \psi_{j}$ as specified in (1), (2) of step 4).
\par
Below we prove that the system $\Sigma_V$ specified in step 6) has the desired properties.
To this end, define
$S := \{x \in X \mid \Phi(x) \in V , f_{\alpha}\varphi_i(x) \in W_{\alpha,i}, h_j \in W_j\}$.
\par
For each $i=1,\ldots,d$, $\pi_i(\Phi(x))=\varphi_i(x)$. Because $\varphi_i \in A_{obs}(\Sigma)$, there exist $k \in \mathbb{N}$, $\alpha = (\alpha_1,\dots ,\alpha_k) \in U^k$ and $j \in \{1,\ldots ,r\}$ such that $\varphi_i(x) = f_{\alpha}h_j(x)$. Notice that by $f_{\alpha_1}h_j$ we denote the Lie derivative of $h_j$ along the vector field $f_{\alpha_1}$, i.e. $L_{f_{\alpha_1}}h_j$, and by $f_{\alpha}h_j$ we denote $L_{f_{\alpha_k}} \cdots L_{f_{\alpha_1}}h_j$. Therefore, for $x\in S$, $\pi_i(\Phi(x))=\varphi_i(x) = f_{\alpha}h_j(x) =
$
$ \sum_{i=1}^n f_{\alpha,i}(x)\frac{d}{dx_i}h_j(x) = \sum_{i=1}^n f_{\alpha,i}(x) \frac{d}{dx_i}(h_j^V(\Phi(x))) $
\noindent
$= \sum_{i=1}^n f_{\alpha,i}(x) \sum_{k=1}^d \frac{d}{dz_k} h_j^V(z)|_{z=\Phi(x)} \frac{d}{dx_i}\Phi_k(x) $
\noindent
$= \sum_{k=1}^d \frac{d}{dz_k} h_j^V(z)|_{z=\Phi(x)}  \sum_{i=1}^n f_{\alpha,i}(x) \frac{d}{dx_i}\Phi_k(x) $
\noindent
$= \sum_{k=1}^d \frac{d}{dz_k} h_j^V(z)|_{z=\Phi(x)}  f_{\alpha}^V(\Phi_k(x)) =
 f_{\alpha}^V h_j^V(\Phi(x))$.
So, $\pi_i$ coincides with $f_{\alpha}^V h_j^V \in A_{obs}(\Sigma_V)$ on $\Phi(S)$.
\par
Let us define $g_i:=f_{\alpha}^V h_j^V$, $i = 1, \ldots,d$.
Suppose there exists a polynomial $Q \in \real[T_1,\ldots,T_d]$ such that $Q(g_1,\ldots ,g_d) = 0$ on $V$. Then, $Q(\pi_1,\ldots ,\pi_d) = 0$ on $\Phi(S)$ and thus $Q(\pi_1(\Phi(x)),\ldots ,\pi_d(\Phi(x))) = Q(\varphi_1(x), \ldots ,\varphi_d(x))= 0$ for all $x \in S$. Because $S$ is a non-empty open set and because $\varphi_1, \ldots ,\varphi_d$ are algebraically independent, it follows that $Q \equiv 0$. This implies that $g_1, \ldots , g_d \in A_{obs}(\Sigma_V)$ are algebraically independent. Hence, transcendence basis of $A_{obs}(\Sigma_V)$ contains at least $d$ elements, i.e. $\trdeg~A_{obs}(\Sigma_V) \geq d = \trdeg~\Na(V)$. This implies that $\trdeg~A_{obs}(\Sigma_V) = \trdeg~\Na(V)$ and thus $\Sigma_V$ is semi-algebraically observable.
\par
Next we prove that $\Sigma_V$ is a local realization of $p_u$.
Let $S_{\Sigma,p}$ be the largest admissible set of inputs such that
$S_{\Sigma,p} \subseteq \widetilde{\U_{pc}} \cap \widetilde{\U_{pc}}(\Sigma)$ and
such that $\forall v \in S_{\Sigma,p} \forall t \in [0,T_v] : x_{\Sigma}(T_u+t;x_0,uv) \in S$.
Since $S$ is open, $S_{\Sigma,p}$ is non-empty.
It is easy to see that for any $v \in S_{\Sigma,p}$,
$x_{\Sigma_V}(T_v;x_0^V,v)$ is well-defined and
$\Phi(x_{\Sigma}(T_{uv};x_0,uv)) = x_{\Sigma_V}(T_v;\Phi(\overline{x}),v)$.
Indeed, $\frac{d}{dt}\Phi(x_{\Sigma}(T_u+t;x_0,uv))$
\\
$= \sum_{k=1}^n \frac{d}{d x_k} \Phi(x_{\Sigma}(T_u+t;x_0,uv)) \frac{d}{d t} (x_{\Sigma}(T_u+t;x_0,uv))_k$
\\
$= \sum_{k=1}^n \frac{d}{d x_k} \Phi(x_{\Sigma}(T_u+t;x_0,uv)) f_{\alpha,k}(x_{\Sigma}(T_u+t;x_0,uv))$
\\
$= \sum_{k=1}^n \frac{d}{d x_k} \Phi(x_{\Sigma}(T_u+t;x_0,uv))$
\\
\begin{flushright}
\vspace{-0.5cm}$f_{\alpha,k}^V(\Phi(x_{\Sigma}(T_u+t;x_0,uv)))$
\end{flushright}
$= f_{\alpha}^V(\Phi(x_{\Sigma}(T_u+t;x_0,uv)))$
\\
and
\\
$x_{\Sigma_V}(0,\Phi(\overline{x}),v) = \Phi(x_{\Sigma}(T_u+0;x_0,uv)) = \Phi(\overline{x}) = x_0^V$.
Hence, $h^V(\Phi(x_{\Sigma}(T_{uv};x_0,uv))=h(x_{\Sigma}(T_{uv};x_0,uv))=p(uv)$
for any $v \in S_{\Sigma,p}$.
Define $\hat{p}:\widetilde{\U_{pc}}^u \cap \widetilde{\U_{pc}}(\Sigma_V) \rightarrow \real^r$
by $\hat{p}(v)=h^V(\Phi(x_{\Sigma}(T_{uv};x_0,uv))$. Then
$(\hat{p}-p)|_{S_{\Sigma,p}}=0$ and then, by Proposition \ref{rem:restrict},
it follows that $\Sigma_V$ is a local realization of $p_u$.
\par
Finally, let us prove that $\Sigma_V$ is semi-algebraically reachable.
To this end, recall the definition of the map $\tau_{\Sigma_V,p_u}^{*}$.
By Corollary \ref{prop2:col}, $\Sigma_V$ is semi-algebraically reachable if and only
if $\tau_{\Sigma_V,p_u}^{*}$ is injective.
Notice that $\tau_{\Sigma_V,p_u}^{*}$ is an algebra morphism.
Let $\hat{p}_u$ be the restriction of $p_u$ to $\widetilde{\U_{pc}}^{u} \cap \widetilde{\U_{pc}}(\Sigma_V)$.
From Proposition \ref{rem:restrict} and Corollary \ref{col1} it follows that
$\trdeg~ A_{obs}(\hat{p}_u)=\trdeg~ A_{obs}(p)=\trdeg~ A_{obs}(\hat{p})$,
where $\hat{p}$ is the restriction of $p$ to
$\widetilde{\U_{pc}} \cap \widetilde{\U_{pc}}(\Sigma)$.
It holds that
$\tau_{\Sigma,p}^{*}(A_{obs}(\Sigma))=A_{obs}(\hat{p})$. Since
$\Sigma$ is semi-algebraically reachable and thus $\tau_{\Sigma,p}^{*}$
is injective, it follows that
$\trdeg ~A_{obs}(\Sigma) = \trdeg ~A_{obs}(\hat{p})$.
From the construction of $\Sigma_V$ one gets that
$\dim \Sigma_V = \trdeg ~\Na(V) = \trdeg ~A_{obs}(\Sigma_V)$. Hence,
$\trdeg ~A_{obs}(\Sigma_V) = \trdeg ~A_{obs}(\hat{p}_u)$.
Because $\tau_{\Sigma_V,p_u}^{*}(A_{obs}(\Sigma_V))=A_{obs}(\hat{p}_u)$, from
\cite[Chapter II, Theorem 28,29]{zariski:samuel:1958} it follows that
$\tau_{\Sigma_V,p_u}^{*}$ is injective.

\end{proof}

%%%%%%%%%%%%%%%%%%%%%%%%%%%%%%%%%%%%%%%%%%%%%%%%%%%%%%%%%%%%%%%%%%%%%%%%%%%%%%%%
%%%%%%%%%%%%%%%%%%%%%%%%%%%%%%%%%%%%%%%%%%%%%%%%%%%%%%%%%%%%%%%%%%%%%%%%%%%%%%%%
\section{Minimal realizations} \label{sec:min}

We say that a local Nash realization $\Sigma$ of a response map $p$ is \emph{minimal} if
$\Sigma$ is of the smallest dimension among all local realizations of $p$.
By $\dim \Sigma$ we refer to the dimension of the state-space of $\Sigma$.

\begin{theorem}
Let $p: \widetilde{\U_{pc}} \rightarrow \real^r$ be a response map.
\\
\noindent
\emph{(i)} $\trdeg~A_{obs}(p) < +\infty \iff \exists u \in \widetilde{\U_{pc}}$ such that $p_u$ has a local realization.
%\\
%\noindent
%\emph{(ii)} If $p_u$ has a local realization for some $u \in \widetilde{\U_{pc}}$,
      %then for every $\epsilon > 0$ there exists $v \in \widetilde{\U_{pc}}$
      %such that $T_v < \epsilon$ and $p_v$ has a local realization.
\\
\noindent
\emph{(ii)} If $\trdeg~A_{obs}(p) = d <+\infty$ then for any $\epsilon > 0$ there exists an input
$u \in \widetilde{\U_{pc}}$ such that $T_u < \epsilon$ and a local Nash realization $\Sigma$ of $p_u$
such that $\dim \Sigma=d$ and $\Sigma$ is semi-algebraically reachable and semi-algebraically observable.
\\
\noindent
\emph{(iii)} Let $\Sigma$ be a local realization of $p_u$. Then, $\Sigma$ is minimal
      $\iff$ $\dim \Sigma = \trdeg~A_{obs}(p)$ $\iff$ $\Sigma$ is semi-algebraically
      reachable and semi-algebraically observable.
\\
\noindent
\emph{(iv)}
Let $\Sigma_1=(X_1,f^1,h^1,x_0^1)$ and $\Sigma_2=(X_2,f^2,h^2,x_0^2)$ be local realizations of $p$ which are both semi-algebraically reachable and semi-algebraically observable. For any $\epsilon >0$ there exist $u \in \widetilde{\U_{pc}}$ such that $T_u<\epsilon$ and there exist open semi-algebraic sets $V_1 \subseteq X_1$, $V_2 \subseteq X_2$ such that $\Sigma_{V_1}=(V_1,f^{V_1},h^{V_1},x_0^{V_1})$ and $\Sigma_{V_2}=(V_2,f^{V_2},h^{V_2},x_0^{V_2})$ are local Nash realizations of $p_u$ which are isomorphic. Here $f^{V_i}$ and $h^{V_i}$ denote restrictions of $f^i$ and $h^i$ respectively, and $x_0^{V_i} = x_{\Sigma_i}(T_u; x_0^i,u)$.
\end{theorem}

\begin{proof}
\\
\noindent
\emph{(i)} The implication "$\implies$" is proven in \cite[Theorem 5.10]{nemcova:petreczky:schuppen:2013}.
For the statement of \cite[Theorem 5.10]{nemcova:petreczky:schuppen:2013} see Theorem \ref{thm:5.10}.
The opposite implication follows from Theorem \ref{thm:5.8} and Corollary \ref{col1}.
%\\
%\noindent
%\emph{(ii)} Let $u \in \widetilde{\U_{pc}}$ be such that $p_u$ has a local realization. From \cite[Theorem 5.8]{nemcova:petreczky:schuppen:2013} and Remark \ref{rem:restrict}, $\trdeg~A_{obs}(p) < +\infty$. The proof is completed by applying \cite[Theorem 5.10]{nemcova:petreczky:schuppen:2013}.
\\
\noindent
\emph{(ii)}
If $\trdeg~A_{obs}(p) = d <+\infty$ then, from Theorem \ref{thm:5.10}, for any $\epsilon>0$ one obtains a local realization $\Sigma_u$ of $p_u$ for some $u \in \widetilde{\U_{pc}}$ such that
$\dim \Sigma_u = d=\trdeg~ A_{obs}(p_u|_{\widetilde{\U_{pc}}^u \cap \widetilde{\U_{pc}}(\Sigma)})=\trdeg~ A_{obs}(p_u)$.
Finally, from Theorem \ref{thm:5.14} and from Proposition \ref{prop2} it follows that
$\Sigma_u$ is semi-algebraically reachable and semi-algebraically observable.
Notice that we used Proposition \ref{rem:restrict}.
%By applying Procedure \ref{proc:reach} to $\Sigma_u$, we construct a semi-algebraically reachable local realization $\Sigma_v$ of $p_{uv}$, where $uv \in \widetilde{\U_{pc}}(\Sigma_u)$. Consequently, by applying Procedure \ref{proc:obs} to $\Sigma_v$, we construct a local realization $\Sigma$ of $p_{uvw}$ for some $uvw \in \widetilde{\U_{pc}}(\Sigma_u)$ such that it is semi-algebraically observable and semi-algebraically reachable. Note that $uvw$ can be chosen arbitrarily small. Moreover, because $\Sigma$ is semi-algebraically reachable and observable local realization of $p_{uvw}$, \cite[Theorem 5.14]{nemcova:petreczky:schuppen:2013} and Remark \ref{rem:restrict} imply that $\dim \Sigma =\trdeg ~A_{obs}(p_{uvw})$. Finally, from Corollary \ref{col1}, $\dim \Sigma =\trdeg ~A_{obs}(p)$.
\\
\noindent
\emph{(iii)} The second equivalence follows directly from Theorem \ref{thm:5.14} and Corollary \ref{col1}. The fact that $\dim \Sigma = \trdeg~A_{obs}(p)$ implies that $\Sigma$ is minimal is due to Theorem \ref{thm:5.13} and Corollary \ref{col1}. Let us prove the opposite implication. Let $\Sigma$ be as in the statement. From \emph{(i)} it follows that $\trdeg~A_{obs}(p_u) <\infty$. Then \emph{(ii)} gives us the existence of a local Nash realization $\Sigma'$ of $p_{uv}$ for some $v\in \widetilde{\U_{pc}}$ such that $\dim \Sigma' = \trdeg~A_{obs}(p_u)$. According to Corollary \ref{col1} $\dim \Sigma' = \trdeg ~A_{obs}(p_{uv})= \trdeg~A_{obs}(p)$. By Theorem \ref{thm:5.13}, this implies that $\Sigma'$ is minimal. Because $\Sigma$ was also minimal, it holds that $\dim \Sigma = \dim \Sigma' = \trdeg~A_{obs}(p)$.
\\
\noindent
\emph{(iv)}
Because $\Sigma_1$ and $\Sigma_2$ are semi-algebraically reachable realizations of $p$,
the maps
$\tau^*_{\Sigma_1,p}: A_{obs}(\Sigma_1) \subseteq \Na(X_1) \rightarrow A_{obs}(p|_{S_1})$ and
$\tau^*_{\Sigma_2,p}: A_{obs}(\Sigma_2) \subseteq \Na(X_2) \rightarrow A_{obs}(p|_{S_2})$ are injective, see Proposition \ref{prop:7.1},
where $S_i=\widetilde{\U_{pc}} \cap \widetilde{\U_{pc}}(\Sigma_i)$, $i=1,2$.
Recall from Proposition \ref{rem:restrict} that
$A_{obs}(p|_{S_i})$ is isomorphic to $A_{obs}(p)$, $i=1,2$.
Let $S=S_1 \cap S_2$.
Consider the algebra $A$ generated by  the restrictions to $S$ of the elements of $\tau^{*}_{\Sigma_1,p}(\Na(X_1))$ and $\tau^{*}_{\Sigma_2,p}(\Na(X_2))$. It is easy to see that $A_{obs}(p|_{S})$ is the restriction to $S$ of the elements of
$\tau^{*}_{\Sigma_i,p}(A_{obs}(\Sigma_i))$, $i=1,2$ and thus $A_{obs}(p|_{S})$ is a subalgebra of $A$. Moreover, since $\Sigma_1$ and $\Sigma_2$ are semi-algebraically observable and
thus $\Na(X_i)$ is algebraic over $A_{obs}(\Sigma_i)$ for $i=1,2$, it follows that
the restrictions to $S$ of elements of
$\tau^{*}_{\Sigma_i,p}(\Na(X_i))$, $i=1,2$ is algebraic over $A_{obs}(p|_{S})$.

Let $\pi_j^{X_1}$ be the $j$th coordinate function on $X_1$,
and let $\pi_j^{X_2}$ be the $j$th coordinate function on $X_2$.
Let $\varphi^{X_i}_j=\tau_{\Sigma_i,p}^{*}(\pi^{X_i}_j)$, $i=1,2$, $j=1,\ldots,n$.
Then both $\varphi_1^{X_1},\ldots,\varphi_n^{X_1}$ and
$\varphi_1^{X_2},\ldots,\varphi_n^{X_2}$ are transcendence basis of $A$.
Thus, for all $j \in \{1,\ldots, n\}$ there exists a non-zero polynomial
$Q_j \in \real[T,T_1,\ldots,T_{n}]$ such that
$Q_j(\varphi_j^{X_2}, \varphi_1^{X_1} , \ldots, \varphi_{n}^{X_1}) = 0$.
Let $Q_j$ be such polynomials of minimal degree with respect to $T$, then
$\frac{d}{dT}Q_j(\varphi_j^{X_2}, \varphi_1^{X_1} , \ldots, \varphi_{n}^{X_1}) \neq 0$.
Similarly, we can choose non-zero polynomials $R_j \in \real[T,T_1,\ldots,T_n]$ such that
$R_j(\varphi_j^{X_1},\varphi_1^{X_2},\ldots,\varphi_n^{X_2})=0$ and
$\frac{d}{dT} R_j(\varphi_j^{X_1},\varphi_1^{X_2},\ldots,\varphi_n^{X_2}) \ne 0$.
In particular, it can be shown that
there exists an input $u$ such that
$\frac{d}{dT}Q_j(\varphi_j^{X_2}, \varphi_1^{X_1} , \ldots, \varphi_{n}^{X_1})(u) \neq 0$,
$\frac{d}{dT}R_j(\varphi_j^{X_1}, \varphi_1^{X_2} , \ldots, \varphi_{n}^{X_2})(u) \neq 0$.
Moreover, due to analyticity of response maps with respect to the switching
times, for any $\epsilon > 0$ we can choose $u$ such that $T_u < \epsilon$.
Indeed, if we define
$S_1:=\frac{d}{dT}Q_1(\varphi_1^{X_2},\varphi_1^{X_1}, \ldots, \varphi_n^{X_1}) \cdots \frac{d}{dT} Q_n(\varphi_n^{X_2},\varphi_1^{X_1},\ldots,\varphi_n^{X_1})$,
and
$S_2:=\frac{d}{dT}R_1(\varphi_1^{X_1},\varphi_1^{X_2}, \ldots, \varphi_n^{X_2}) \cdots \frac{d}{dT} R_n(\varphi_n^{X_1},\varphi_1^{X_2},\ldots$
$\ldots,\varphi_n^{X_2})$, then $S_1S_2$
belongs to $A$. Because $A$ is an integral domain and because $S_1S_2 \ne 0$, it holds that $S_1S_2(u) \ne 0$ for some input $u = (\alpha_1,t_1)\cdots (\alpha_k,t_k)$, where
$t_1,\ldots,t_k \in [0,+\infty)$, $\alpha_1,\ldots,\alpha_k \in U$. But then,
$\frac{d}{dT}Q_j(\varphi_j^{X_2}, \varphi_1^{X_1} , \ldots, \varphi_{n}^{X_1})(u) \neq 0$ and
$\frac{d}{dT}R_j(\varphi_j^{X_1}, \varphi_1^{X_2} , \ldots, \varphi_{n}^{X_2})(u) \neq 0$
for all $j=1,\ldots,n$.
Moreover, since $S_1S_2(u)$ is analytic in $t_1,\ldots,t_k$,
we can choose $t_1,\ldots,t_k$ to be arbitrary small, i.e. $t_1+\cdots+t_k < \epsilon$ for arbitrary $\epsilon>0$.
\par
Define $x^1=(\varphi^{X_1}_1(u),\ldots,\varphi^{X_1}_n(u))$ and
$x^2=(\varphi^{X_2}_1(u),\ldots,\varphi^{X_2}_n(u))$.
By applying the implicit function theorem to the maps $Q_j(\varphi_j^{X_2}, \varphi_1^{X_1} , \ldots, \varphi_{n}^{X_1})$,
it follows that there exist an open semi-algebraic neighborhood
$W_1^{1}$ of $x^1$, an open semi-algebraic neighborhood $W_2^1$ of $x^2$, and a
Nash map $\xi_1:W_1^1 \rightarrow W_2^1$ such that
$\forall x \in W^1_1, y=(y_1,\ldots,y_n) \in W_2^1:$
$(\forall j=1,\ldots,n: Q_j(y_j,x)=0 ) \iff \xi_1(x)=y$.
Similarly, by applying the implicit function theorem to
$R_j(\varphi_j^{X_1}, \varphi_1^{X_2} , \ldots, \varphi_{n}^{X_2})$, $j=1,\ldots,n$,
it follows that there exist open connected semi-algebraic sets
$W_1^2,W_2^2$ such that $x^1 \in W_1^2$, $x^2 \in W_2^2$, and a Nash map
$\xi_2:W_2^{2} \rightarrow W_1^2$ such that
$\forall y \in W^2_2, x=(x_1,\ldots,x_n) \in W_2^2:$
$(\forall j=1,\ldots,n: R_j(x_j,y)=0 ) \iff \xi_2(y)=x$.
Set $W_2=W_2^2 \cap W_2^1 \cap X_2$ and set
$W_1=X_1 \cap \xi_1^{-1}(W_2)$.
It then follows that $x^1 \in W_1$, $x^2 \in W_2$, and $W_1,W_2$ are open semi-algebraic
sets.
Consider the map $\psi(x)=\xi_2(\xi_1(x))-x: W_1 \rightarrow \real^n$.
By applying Proposition \ref{prop2} to $\widetilde{\U_{pc}}=\{ v \in \widetilde{\U_{pc}}(\Sigma_1) \cap \widetilde{\U_{pc}}(\Sigma_2) \mid x_{\Sigma_i}(T_{v};x_0^i,v) \in W_i, i=1,2\}$ and using reachability of
$\Sigma_1$,
it follows that $\psi=0$.

Hence,
$\xi_2(\xi_1(x))=x$ for any $x \in W_1$.
Therefore $\xi_1:W_1 \rightarrow W_2$
is injective and $D\xi_2(\xi_1(x))D\xi_1(x)=I_n$ for all $x \in W_1$.
In particular, $D\xi_2(x^2)D\xi_1(x^1)=I_n$, i.e. the Jacobian of $\xi_1$ at $x^1$
is invertible. Then by inverse function theorem, see \cite{bochnak:coste:roy:1998},
there exist semi-algebraic neighborhoods $V_1 \subseteq W_1$ of $x^1$ and
$V_2 \subseteq W_2$ of $x^2$ such that the restriction $\xi_1|_{V_1}:V_1 \rightarrow V_2$
is a Nash diffeomorphism.
Define the systems $\Sigma_{V_i}=(V_i,f^{V_i},h^{V_i},x^i)$, $i=1,2$, where
$f_{\alpha}^{V_i}$ is the restriction of $f_{\alpha}$ to $V_i$
and $h^{V_i}$ is the restriction of $h$ to $V_i$.
Then  $\phi=\xi_1|_{V_1}$ is the desired isomorphism from $\Sigma_{V_1}$ to $\Sigma_{V_2}$.
\par
We already know that $\phi$ is a Nash diffeomorphism. From
the discussion above, $\phi(x_{\Sigma_{V_1}}(T_v;x_1,v))=x_{\Sigma_{V_2}}(T_v;x_2,v))$
and $h^{V_2}(\phi(x_{\Sigma_{V_1}}(T_v;x_1,v))=h^{V_2}(x_{\Sigma_{V_2}}(T_v;x_2,v))=p(uv)=h^{V_1}(x_{\Sigma_{V_1}}(T_v;x_1,v))$
for all $v \in \widetilde{\U_{pc}}(\Sigma_{V_1}) \cap \widetilde{\U_{pc}}(\Sigma_{V_2})$.
From Proposition \ref{prop2},  by taking $\widetilde{\U_{pc}}=\widetilde{\U_{pc}}(\Sigma_{V_1}) \cap \widetilde{\U_{pc}}(\Sigma_{V_2})$,
it then follows that $f_{\alpha}^{V_2}(\phi(x))=D_{\alpha} \phi(x) f_{\alpha}^{V_1}(x)$
and $h^{V_2}(\phi(x))=h^{V_1}(x)$
for every $x$ which is reachable in $\Sigma_{V_1}$. Since  $\Sigma_{V_1}$ is
semi-algebraically reachable and all the involved maps are Nash functions,
$f^{V_2}_{\alpha} \circ \phi = f^{V_1}_{\alpha}$ and
$h^{V_2} \circ \phi = h^{V_1}$.
From the definition of $\phi$ it follows that $\phi(x_1)=x_2$.
That is, $\phi$ is indeed an isomorphism from $\Sigma_{V_1}$ to $\Sigma_{V_2}$.
\end{proof}

%%%%%%%%%%%%%%%%%%%%%%%%%%%%%%%%%%%%%%%%%%%%%%%%%%%%%%%%%%%%%%%%%%%%%%%%%%%%%%%%
%%%%%%%%%%%%%%%%%%%%%%%%%%%%%%%%%%%%%%%%%%%%%%%%%%%%%%%%%%%%%%%%%%%%%%%%%%%%%%%%
\section*{Conclusions}

The paper provides constructive procedures, not yet algorithms, for reachability and observability reduction of local
Nash realizations of shifted response maps. Further, the characterization of minimality
and existence of local Nash realizations is given. More research is required to extend
the obtained results for global Nash realizations and for applying them to related problems, e.g.
in system identification or model reduction.

%%%%%%%%%%%%%%%%%%%%%%%%%%%%%%%%%%%%%%%%%%%%%%%%%%%%%%%%%%%%%%%%%%%%%%%%%%%%%%%%

\appendix
Below we state the theorems proven in \cite{nemcova:petreczky:schuppen:2013}.

\vspace{0.2cm}
\begin{theorem} \label{thm:5.8}
\cite[Theorem 5.8]{nemcova:petreczky:schuppen:2013}
Let $\Sigma = (X,f,h,x_0)$ be a Nash realization of a response
map $p: \widetilde{\U_{pc}} \rightarrow \real^{r}$. Then, $\trdeg
~A_{obs}(p) \leq \dim (X)$. Hence, if there exists a Nash
realization of a response map $p$ then $\trdeg ~A_{obs}(p) <
+\infty$.
\end{theorem}

\vspace{0.2cm}
\begin{theorem} \label{thm:5.10}
\cite[Theorem 5.10]{nemcova:petreczky:schuppen:2013}
Assume $\trdeg~A_{obs}(p) = d <+\infty$ and $U$ is finite. For any $\epsilon >0$ one
can choose an input $u \in \widetilde{\U_{pc}}$ and a Nash system $\Sigma$ such that
$T_u < \epsilon$ and $\Sigma$ is a local realization of $p_u$.
\end{theorem}

\vspace{0.2cm}
\begin{proposition} \label{prop:7.1}
\cite[Proposition 7.1]{nemcova:petreczky:schuppen:2013}
Let $\Sigma = (X,f,h,x_0)$ be a Nash realization of a response
map $p: \widetilde{\U_{pc}} \rightarrow \real^r$. Consider the following
statements.
\begin{itemize}
 \item[(i)] $\Sigma$ is semi-algebraically reachable,
 \item[(ii)] the dual input-to-state map $\tau_{\Sigma}^* : \Na(X) \rightarrow \A(\widetilde{\U_{pc}} \rightarrow \real)$ is injective,
 \item[(iii)] the ideal $\Ker ~\tau_{\Sigma}^*$ is the zero ideal in $\Na(X)$,
 \item[(iv)] $\trdeg ~A_{obs}(p) = \trdeg ~A_{obs}(\Sigma)$.
\end{itemize}
The statements (i), (ii) and (iii) are equivalent. The statement (ii) implies (iv). If
$\Sigma$ is semi,-algebraically observable, then (iv) and (ii) are equivalent.
\end{proposition}

\vspace{0.2cm}
\begin{theorem} \label{thm:5.13}
\cite[Theorem 5.13]{nemcova:petreczky:schuppen:2013}
If the dimension of a Nash realization $\Sigma$
of a response map $p$ equals $\trdeg ~ A_{obs}(p)$,
 then $\Sigma$ is a minimal Nash realization of $p$.
\end{theorem}

\vspace{0.2cm}
\begin{theorem} \label{thm:5.14}
\cite[Theorem 5.14]{nemcova:petreczky:schuppen:2013}
A Nash realization $\Sigma$ of a response map $p$ is
semi-algebraically reachable and semi-algebraically observable
if and only if $\dim (\Sigma) = \trdeg~A_{obs}(p)$.
\end{theorem}

\vspace{0.2cm}
\begin{theorem} \label{thm:5.15}
\cite[Theorem 5.15]{nemcova:petreczky:schuppen:2013}
If a Nash realization $\Sigma$ of a response map $p$ is
semi-algebraically reachable and semi-algebraically observable
then $\Sigma$ is minimal. In particular, if $\Sigma$ is semi-algebraically reachable and strongly
semi-algebraically observable, then it is minimal.
\end{theorem}

\end{document}